\documentclass[11pt]{article}

\usepackage{amsmath,amssymb,amsthm,amsfonts}
\usepackage{geometry}
\usepackage{enumitem}
\usepackage[hidelinks]{hyperref}

\theoremstyle{plain}
\newtheorem{theorem}{Theorem}[section]
\newtheorem{lemma}[theorem]{Lemma}
\newtheorem{corollary}[theorem]{Corollary}
\newtheorem{proposition}[theorem]{Proposition}

\theoremstyle{definition}
\newtheorem{definition}[theorem]{Definition}

\theoremstyle{remark}
\newtheorem{remark}[theorem]{Remark}

\numberwithin{equation}{section}

\newcommand{\TN}{\mathrm{TN}}

\newcommand{\diag}{\operatorname{diag}}
\newcommand{\supp}{\operatorname{supp}}
\newcommand{\e}{\mathrm e}

\title{Entrywise Positivity Preservers on Green Matrices}
\author{Wei Xie}
\date{}

\begin{document}
	
	\maketitle
	
\begin{abstract}
We classify entrywise functions preserving positive semidefiniteness on
Green matrices with positive parameters. For each fixed
\(n\ge2\), PSD preservation is equivalent to \(f\) being nondecreasing and
\((f(tr^{i+j}))_{i,j=0}^{n-1}\) being PSD for all \(t>0,r>1\).
For \(f\in C^{2n-2}\), this is equivalent to \(f'\ge0\) and positive
semidefiniteness of \(((t\,d/dt)^{i+j}f(t))_{i,j=0}^{n-1}\) at every
positive \(t\). These fixed-order classes strictly decrease with \(n\); their
intersection consists exactly of
\(f(t)=\int_{[0,\infty)}t^\alpha\,d\mu(\alpha)\), where \(\mu\) is a finite
positive measure and the integral is finite for every \(t>0\).
At each fixed order, entrywise Loewner monotonicity is equivalent to
\(f\in C^1\) and \(f'\) belonging to the corresponding PSD preserver
class. In all orders, monotonicity itself implies smoothness. Nonnegative minors of orders one and two on every strict
three-by-three Green input already force \(f\equiv0\) or
\(f(t)=ct^\alpha\), with \(c>0\) and \(\alpha\ge0\). They give TN output in every order; the nonzero members also preserve PSD
Green structure. We also classify
PSD, TN, and Loewner-monotone extensions to the closure allowing zero
entries. No regularity is assumed except in the differential criterion.
\end{abstract}


	\medskip
	
	
	\section{Introduction}\label{sec:introduction}
	
	We study entrywise functions on the class of positive-parameter discrete
	Green matrices
	\begin{equation}\label{eq:intro-green}
		G(p,q)
		=
		\bigl(g_{ij}\bigr)_{i,j=1}^n,
		\qquad
		g_{ij}
		=
		p_{\min(i,j)}q_{\max(i,j)},
		\qquad
		p_i,q_i>0.
	\end{equation}
	Equivalently,
	\begin{equation}\label{eq:intro-green-entry}
		g_{ij}
		=
		\begin{cases}
			p_iq_j,&i\le j,\\
			p_jq_i,&j<i.
		\end{cases}
	\end{equation}
	Such matrices are discrete analogues of Green kernels for second-order
	boundary-value problems~\cite{OlshevskyStrangZhlobich2010}.  Their positive definite members have symmetric irreducible tridiagonal
	\(M\)-matrices as inverses and are therefore inverse \(M\)-matrices
	\cite{Markham1972,McDonaldNabbenNeumannSchneiderTsatsomeros1998,Nabben2001}.
	Their relation to oscillation theory and total positivity is classical; see
	\cite{GantmacherKrein,Karlin,Pinkus}.
	For the matrices in \eqref{eq:intro-green}, the monotonicity of the ratios
	\(p_i/q_i\) controls both positive semidefiniteness and total
	nonnegativity.
	
	Several algebraic and numerical properties of Green matrices have been
	studied.  Pinkus~\cite[Theorem~4.2]{Pinkus} gives the
	total-nonnegativity criterion and formulas governing the minors.  Carmona,
	Encinas, and Mitjana~\cite{CarmonaEncinasMitjana2015} use Green matrices in
	the analysis of generalized linear polyominoes.  Olshevsky, Strang, and
	Zhlobich~\cite{OlshevskyStrangZhlobich2010} study
	characteristic-polynomial recurrences, structured factorizations, and
	fast inversion.  Delgado, Pe\~na, and Pe\~na derive formulas and algorithms
	for arbitrary minors, obtain bidiagonal decompositions for Green matrices
	and their Hadamard products, and use these decompositions for high-relative-
	accuracy computations
	\cite{DelgadoPenaPena2023Axioms,DelgadoPenaPena2023AML}.  Their subsequent
	work extends the bidiagonal and numerical analysis to nonsymmetric
	generalized Green matrices \cite{DelgadoPenaPena2024}.  Guillot and
	Wu~\cite{GuillotWu2019} proved that a GCD matrix is totally nonnegative if
	and only if it is a Green matrix, and gave sufficient conditions under
	which \(h\circ\gcd\) remains totally nonnegative.  Their results concern
	the arithmetic GCD subclass, whereas we consider arbitrary
	positive-parameter Green matrices.

	We consider functions \(f:(0,\infty)\to\mathbb R\).  For a matrix
	\(A=(a_{ij})\) with positive entries, write
	\begin{equation}\label{eq:entrywise-map}
		f[A]=\bigl(f(a_{ij})\bigr).
	\end{equation}
	We ask first for all functions satisfying
	\begin{equation}\label{eq:PSD-only-problem}
		G\succeq0,\quad G\text{ a positive-parameter Green matrix}
		\quad\Longrightarrow\quad
		f[G]\succeq0.
	\end{equation}
	Condition~\eqref{eq:PSD-only-problem} does not require \(f[G]\) to be a
	Green matrix. 	For example,
	\[
	f(t)=1+t
	\]
	does not generally preserve Green structure.  Nevertheless, for every
	positive semidefinite matrix \(G\),
	\[
	f[G]=\mathbf1\mathbf1^T+G\succeq0,
	\]
	where \(\mathbf{1}\) denotes the all-ones column vector, so \(\mathbf{1}\mathbf{1}^T\) is the all-ones matrix.
	Hence preservation
	of positive semidefiniteness on the Green class does not imply preservation
	of Green structure.
	
	Related classification results are known for positive semidefinite matrices,
	Hankel matrices and kernels, moment sequences, and totally nonnegative
	matrices; see
	\cite{FitzGeraldHorn,GuillotKhareRajaratnam,
		BeltonGuillotKharePutinar2022,
		BeltonGuillotKharePutinar2023}.  Belton, Guillot, Khare, and Putinar
	classify transforms of moment sequences and totally nonnegative Hankel
	matrices in \cite{BeltonGuillotKharePutinar2022}.  In
	\cite{BeltonGuillotKharePutinar2023} they study continuous Hankel kernels,
	P\'olya frequency functions and sequences, and more general TN and TP
	kernels.  Their Proposition~7.4 identifies continuous PSD Hankel kernels
	with exponential moments; this representation is used in
	Section~\ref{sec:PSD-only}.  Here the inputs are restricted to the Green
	class, while the images are required only to be positive semidefinite.
	Hiai~\cite{Hiai2009} proved that, for functions on a symmetric interval
about zero and matrix orders at least three, entrywise Loewner monotonicity
on the full PSD cone with entries in that interval is equivalent to
differentiability and PSD preservation by the derivative. We consider
Loewner comparison within fixed-parameter Green families and between Green
matrices with arbitrary positive parameters.
	
	Positive Hadamard powers of a Green matrix satisfy
	\[
	G(p,q)^{\circ\alpha}
	=
	G(p^\alpha,q^\alpha),
	\qquad \alpha>0.
	\]
	Thus positive semidefinite Green matrices are infinitely divisible in the
	entrywise sense; see Bhatia~\cite{Bhatia2006} for background.  For
	comparison, Dellacherie,
	Martinez, and San Martin
	\cite{DellacherieMartinezSanMartin2009} proved that the Hadamard
	\(\alpha\)-power of every inverse \(M\)-matrix remains an inverse
	\(M\)-matrix for \(\alpha\ge1\).  We require only PSD output on the smaller
	Green class, not preservation of the inverse-\(M\)-matrix structure.  These
	results motivate the PSD-preserver problem \eqref{eq:PSD-only-problem}.

	Besides \eqref{eq:PSD-only-problem}, we consider Loewner monotonicity,
	totally nonnegative output,
	positive definite output on strict Green matrices, and preservation of
	Green structure, with the ratios \(p_i/q_i\) required to be nondecreasing
	or strictly increasing.  We impose no a priori regularity on \(f\).  We also ask which
	matrix orders are needed to determine the corresponding all-dimensions
	preserver class, and how these requirements change when zero
	entries are allowed by taking the closure of the input class.
	
	Section~\ref{sec:preliminaries} records the classical Green-matrix and
	Hankel-kernel tools used below.  Section~\ref{sec:PSD-only} treats PSD
	output, and Section~\ref{sec:Loewner-fixed-q} treats Loewner monotonicity.
	Sections~\ref{sec:TN-output} and \ref{sec:strict-PD} consider TN and
	positive-definite output.  Section~\ref{sec:structure-preservers} considers entrywise maps
	whose images remain Green matrices. Section~\ref{sec:zero} treats zero entries, and
	Section~\ref{sec:conclusion} gives concluding remarks.
	
	\section{Terminology and preliminaries}\label{sec:preliminaries}
	
	All matrices are real.  For a symmetric matrix \(A\), we write
	\[
	A\succeq0
	\quad\text{and}\quad
	A\succ0
	\]
	for positive semidefiniteness and positive definiteness, respectively. 
	For symmetric matrices \(A\) and \(B\) of the same order, we write
	\[
	A\succeq B
	\]
	if \(A-B\succeq0\). This partial order is called the Loewner order.

	 If
	all entries of \(A=(a_{ij})\) are nonnegative, we write
	\[
	A\ge_{\mathrm e}0.
	\]
	
	A matrix is \emph{totally nonnegative} (TN) if all its minors are
	nonnegative, and \emph{totally positive} (TP) if all its minors are
	positive. We write \(\TN_2\) for matrices whose minors of orders one and
	two are all nonnegative.

	For a matrix \(A=(a_{ij})\) with positive entries and \(\alpha\ge0\), set
	\begin{equation}\label{eq:Hadamard-power}
		A^{\circ\alpha}
		=
		\bigl(a_{ij}^{\alpha}\bigr),
	\end{equation}
	with the convention
	\begin{equation}\label{eq:zero-Hadamard-power}
		A^{\circ0}
		=
		\mathbf1\mathbf1^T.
	\end{equation}
	
	\subsection{Continuous positive semidefinite Hankel kernels}
	
	Let \(X\subseteq\mathbb R\) be an interval, and write
	\[
	X+X:=\{x+y:x,y\in X\}.
	\]
	A kernel \(K:X\times X\to\mathbb R\) is called a Hankel kernel if there
	exists a function
	\[
	g:X+X\to\mathbb R
	\]
	such that
	\begin{equation}\label{eq:Hankel-kernel}
		K(x,y)=g(x+y),
		\qquad x,y\in X.
	\end{equation}
	We refer to \(g\) as the generating function of \(K\). The kernel \(K\) is
	called positive semidefinite if, for every \(n\ge1\) and every
	\(x_1,\ldots,x_n\in X\), the matrix
	\[
	\bigl(K(x_i,x_j)\bigr)_{i,j=1}^n
	=
	\bigl(g(x_i+x_j)\bigr)_{i,j=1}^n
	\]
	is positive semidefinite.

	We use the following representation.

	\begin{lemma}[{\cite[Proposition~7.4]{BeltonGuillotKharePutinar2023}}]
		\label{lem:Hankel-representation}
		Let \(X\subseteq\mathbb R\) be an open interval, and let
		\(K:X\times X\to\mathbb R\) be a Hankel kernel with generating function
		\(g:X+X\to\mathbb R\), so that
		\[
		K(x,y)=g(x+y),
		\qquad x,y\in X.
		\]
		Suppose that \(K\) is continuous on \(X\times X\). Then \(K\) is
		positive semidefinite if and only if there exists a positive Borel measure
		\(\mu\) on \(\mathbb R\) such that
		\begin{equation}\label{eq:Hankel-representation}
			K(x,y)
			=
			\int_{\mathbb R}
			\e^{(x+y)\alpha}\,d\mu(\alpha),
			\qquad x,y\in X,
		\end{equation}
		and the displayed integral is finite for all \(x,y\in X\).
	\end{lemma}

	\subsection{Basic factorizations of Green matrices}
	
	\begin{definition}\label{def:green}
		Let
		\[
		p=(p_1,\ldots,p_n),
		\qquad
		q=(q_1,\ldots,q_n)
		\]
		be positive sequences.  Define
		\begin{equation}\label{eq:green-definition}
			G(p,q)
			=
			\bigl(g_{ij}\bigr)_{i,j=1}^n,
			\qquad
			g_{ij}
			=
			p_{\min(i,j)}q_{\max(i,j)}.
		\end{equation}
	\end{definition}
	
	Set
	\begin{equation}\label{eq:green-ratio}
		r_i=\frac{p_i}{q_i},
		\qquad i=1,\ldots,n,
	\end{equation}
	and
	\begin{equation}\label{eq:green-increments}
		r_0=0,
		\qquad
		\Delta_i=r_i-r_{i-1},
		\qquad i=1,\ldots,n.
	\end{equation}
	Write
	\begin{equation}\label{eq:green-Dq}
		D_q=\diag(q_1,\ldots,q_n),
	\end{equation}
	and let \(L=(\ell_{ij})\) be the lower triangular summation matrix
	\begin{equation}\label{eq:green-L}
		\ell_{ij}
		=
		\begin{cases}
			1,&i\ge j,\\
			0,&i<j.
		\end{cases}
	\end{equation}
	
	The following standard factorization fixes notation and gives the
	positivity criterion used below.
	
	\begin{proposition}
		\label{prop:green-factorization}
		For every positive-parameter Green matrix \(G(p,q)\),
		\begin{equation}\label{eq:green-congruence}
			G(p,q)
			=
			D_q A_{\min}(r)D_q,
		\end{equation}
		where
		\[
		A_{\min}(r)
		=
		\bigl(r_{\min(i,j)}\bigr)_{i,j=1}^n.
		\]
		Moreover,
		\begin{equation}\label{eq:green-LDL}
			G(p,q)
			=
			D_qL
			\diag(\Delta_1,\ldots,\Delta_n)
			L^TD_q.
		\end{equation}
	\end{proposition}
	
	\begin{proof}
		For all \(i,j\), \eqref{eq:green-ratio} gives
		\begin{align*}
			\bigl(D_qA_{\min}(r)D_q\bigr)_{ij}
			&=
			q_iq_jr_{\min(i,j)}\\
			&=
			q_iq_j
			\frac{p_{\min(i,j)}}{q_{\min(i,j)}}\\
			&=
			p_{\min(i,j)}q_{\max(i,j)}\\
			&=
			g_{ij}.
		\end{align*}
		This proves \eqref{eq:green-congruence}.
		
		Also, \eqref{eq:green-increments} gives
		\[
		r_m=\sum_{k=1}^m\Delta_k,
		\]
		and hence
		\begin{align*}
			\left(
			L\diag(\Delta_1,\ldots,\Delta_n)L^T
			\right)_{ij}
			&=
			\sum_{k=1}^{\min(i,j)}\Delta_k\\
			&=
			r_{\min(i,j)}.
		\end{align*}
		Thus
		\[
		A_{\min}(r)
		=
		L\diag(\Delta_1,\ldots,\Delta_n)L^T.
		\]
		Substitution into \eqref{eq:green-congruence} proves
		\eqref{eq:green-LDL}.
	\end{proof}
	
\begin{corollary} \label{cor:green-positivity-criterion}
		Let \(G(p,q)\) be a positive-parameter Green matrix.  Then
		\begin{equation}\label{eq:green-PSD-equivalence}
			G(p,q)\succeq0
			\quad\Longleftrightarrow\quad
			0<r_1\le r_2\le\cdots\le r_n,
		\end{equation}
		and
		\begin{equation}\label{eq:green-PD-equivalence}
			G(p,q)\succ0
			\quad\Longleftrightarrow\quad
			0<r_1<r_2<\cdots<r_n.
		\end{equation}
\end{corollary}	
	\begin{proof}
		By \eqref{eq:green-LDL}, \(G(p,q)\) is congruent to
		\(\diag(\Delta_1,\ldots,\Delta_n)\).  Therefore
		\[
		G(p,q)\succeq0
		\quad\Longleftrightarrow\quad
		\Delta_i\ge0
		\quad(i=1,\ldots,n),
		\]
		and
		\[
		G(p,q)\succ0
		\quad\Longleftrightarrow\quad
		\Delta_i>0
		\quad(i=1,\ldots,n).
		\]
		Since \(r_1>0\), these conditions are equivalent, respectively, to
		\[
		0<r_1\le\cdots\le r_n
		\]
		and
		\[
		0<r_1<\cdots<r_n.
		\]
	\end{proof}
	
	\begin{remark}\label{rem:green-TN-criterion}
		In our terminology, Pinkus~\cite[Theorem~4.2]{Pinkus} proves that
		\[
		G(p,q)\text{ is TN}
		\quad\Longleftrightarrow\quad
		0<r_1\le\cdots\le r_n.
		\]
		Together with the positivity criterion above, this yields
		\[
		G(p,q)\succeq0
		\quad\Longleftrightarrow\quad
		G(p,q)\text{ is TN}.
		\]
	\end{remark}
	
	\begin{definition}
		\label{def:green-classes}
		For \(n\ge1\), let \(\mathcal G_n^+\) denote the class of all
		positive-parameter positive semidefinite Green matrices of order \(n\).
		By Corollary~\ref{cor:green-positivity-criterion},
		\begin{equation}\label{eq:PSD-Green-class}
			\mathcal G_n^+
			=
			\left\{
			G(p,q):
			0<\frac{p_1}{q_1}
			\le\frac{p_2}{q_2}
			\le\cdots\le\frac{p_n}{q_n}
			\right\}.
		\end{equation}
		
		A positive-parameter Green matrix \(G(p,q)\) is called
		\emph{strict} if
		\begin{equation}\label{eq:strict-Green-ratios}
			0<\frac{p_1}{q_1}
			<\frac{p_2}{q_2}
			<\cdots
			<\frac{p_n}{q_n}.
		\end{equation}
		We denote the class of all strict Green matrices of order \(n\) by
		\(\mathcal G_n^{\mathrm{str}}\). By
		\eqref{eq:green-PD-equivalence}, every strict Green matrix is positive
		definite.
	\end{definition}
\begin{remark}\label{rem:positive-rank-one-green}
	Every positive rank-one matrix \(xx^T\), with \(x_i>0\), belongs to
	\(\mathcal G_n^+\), by taking \(p_i=q_i=x_i\) in
	Definition~\ref{def:green}.
\end{remark}

\begin{proposition}
	\label{prop:green-Hadamard-power}
	Let \(\alpha>0\).
	
	\begin{enumerate}[label=\textup{(\roman*)}]
		\item If \(G(p,q)\in\mathcal G_n^+\), then
		\begin{equation}\label{eq:green-power}
			G(p,q)^{\circ\alpha}
			=
			G(p^\alpha,q^\alpha)
			\in\mathcal G_n^+.
		\end{equation}
		
		\item If \(G(p,q)\in\mathcal G_n^{\mathrm{str}}\), then
		\begin{equation}\label{eq:strict-green-power}
			G(p,q)^{\circ\alpha}
			=
			G(p^\alpha,q^\alpha)
			\in\mathcal G_n^{\mathrm{str}}.
		\end{equation}
	\end{enumerate}
\end{proposition}

\begin{proof}
	The identity
	\[
	G(p,q)^{\circ\alpha}
	=
	G(p^\alpha,q^\alpha)
	\]
	follows entrywise from Definition~\ref{def:green}. The ratio sequence of
	the resulting Green matrix is
	\[
	\frac{p_i^\alpha}{q_i^\alpha}
	=
	\left(\frac{p_i}{q_i}\right)^\alpha,
	\qquad i=1,\ldots,n.
	\]
	Since \(t\mapsto t^\alpha\) is strictly increasing on \((0,\infty)\), it
	preserves both weak and strict inequalities among the ratios. Thus
	\[
	G(p,q)\in\mathcal G_n^+
	\quad\Longrightarrow\quad
	G(p,q)^{\circ\alpha}\in\mathcal G_n^+,
	\]
	and
	\[
	G(p,q)\in\mathcal G_n^{\mathrm{str}}
	\quad\Longrightarrow\quad
	G(p,q)^{\circ\alpha}\in\mathcal G_n^{\mathrm{str}}.
	\]
\end{proof}

\section{PSD preservers on the Green class}\label{sec:PSD-only}

Let
$
f:(0,\infty)\longrightarrow\mathbb R.
$
We ask that \(f[G]\) be positive semidefinite for every positive
semidefinite Green matrix \(G\) of every order, but impose no Green-structure
condition on \(f[G]\).

\begin{theorem}
	\label{thm:PSD-only-preservers}
	Let
	$
	f:(0,\infty)\longrightarrow\mathbb R.
	$
	The following are equivalent.
	\begin{enumerate}[label=\textup{(\roman*)}]
		\item for every \(n\ge1\) and \(G\in\mathcal G_n^+\), one has
		\(f[G]\succeq0\);
		
		\item either \(f\equiv0\), or there is a nonzero finite positive Borel
		measure \(\mu\) on \([0,\infty)\) such that
		\begin{equation}\label{eq:PSD-preserver-representation}
			f(t)
			=
			\int_{[0,\infty)}t^\alpha\,d\mu(\alpha),
			\qquad t>0,
		\end{equation}
		and the integral is finite for every \(t>0\).
	\end{enumerate}
\end{theorem}

The next proposition derives positivity, monotonicity, and continuity from
tests of orders one and two.

\begin{proposition}\label{prop:automatic-regularity}
	If \(f[G]\succeq0\) for every \(G\in\mathcal G_n^+\) and every \(n\ge1\),
	then either \(f\equiv0\), or \(f\)
	is positive, nondecreasing, and continuous on \((0,\infty)\).  In the latter
	case it also satisfies
	\begin{equation}\label{eq:multiplicative-log-convexity}
		f(\sqrt{uv})^2\le f(u)f(v),
		\qquad u,v>0.
	\end{equation}
\end{proposition}

\begin{proof}
	Testing one-by-one matrices first gives \(f(t)\ge0\).  If \(f(t_0)=0\),
	then for arbitrary \(u>0\), preservation on the positive rank-one matrix
	\[
	\begin{pmatrix}
		t_0&\sqrt{t_0u}\\
		\sqrt{t_0u}&u
	\end{pmatrix}
	\]
	forces \(f(\sqrt{t_0u})=0\).  Since \(u\) is arbitrary, \(f\equiv0\).
	Thus every nonzero preserver is positive.
	
	For \(0<a\le b\), the matrix
	\begin{equation}\label{eq:monotonicity-test}
		\begin{pmatrix}a&a\\a&b\end{pmatrix}
	\end{equation}
	belongs to \(\mathcal G_2^+\).  The determinant of its image is
	\(f(a)(f(b)-f(a))\), so positivity of \(f(a)\) gives \(f(a)\le f(b)\).
	This proves monotonicity.  Similarly, applying \(f\) to the positive
	rank-one matrix
	\[
	\begin{pmatrix}u&\sqrt{uv}\\\sqrt{uv}&v\end{pmatrix}
	\]
	and taking its determinant proves
	\eqref{eq:multiplicative-log-convexity}.
	
	Finally, set \(h(s)=\log f(\e^s)\).  Inequality
	\eqref{eq:multiplicative-log-convexity} says that \(h\) is midpoint convex,
	and monotonicity makes \(h\) locally bounded.  Hence \(h\) is continuous,
	and so is \(f\).
\end{proof}

\subsection{Rank-one tests and Hankel kernels}

\begin{proposition}
	\label{prop:real-power-representation}
	Suppose \(f\not\equiv0\) and \(f[G]\succeq0\) for every
	\(G\in\mathcal G_n^+\) and every \(n\ge1\).  There is a nonzero
	finite positive Borel measure \(\mu\) on \(\mathbb R\) such that
	\begin{equation}\label{eq:real-power-representation}
		f(t)
		=
		\int_{\mathbb R}t^\alpha\,d\mu(\alpha),
		\qquad t>0.
	\end{equation}
	Moreover,
	\begin{equation}\label{eq:real-power-integrability}
		\int_{\mathbb R}t^\alpha\,d\mu(\alpha)<\infty,
		\qquad t>0.
	\end{equation}
\end{proposition}

\begin{proof}
	Set \(g(s)=f(\e^s)\).  Proposition~\ref{prop:automatic-regularity}
	shows that \(g\) is continuous.  Given \(s_1,\ldots,s_n\in\mathbb R\), the
	positive rank-one matrix
	\[
	\bigl(\e^{s_i+s_j}\bigr)_{i,j=1}^n
	\]
	is Green.  Its entrywise image is positive semidefinite, so
	\[
	K(s,t)=g(s+t)
	\]
	is a continuous positive semidefinite Hankel kernel on
	\(\mathbb R\times\mathbb R\).  Lemma~\ref{lem:Hankel-representation}
	provides a positive Borel measure \(\mu\) on \(\mathbb R\) such that
	\begin{equation}\label{eq:Hankel-representation-applied}
		g(s+t)
		=
		\int_{\mathbb R}
		\e^{(s+t)\alpha}\,d\mu(\alpha).
	\end{equation}
	Setting \(s=t=u/2\) gives
	\begin{equation}\label{eq:g-bilateral-laplace}
		g(u)
		=
		\int_{\mathbb R}
		\e^{u\alpha}\,d\mu(\alpha),
		\qquad u\in\mathbb R.
	\end{equation}
	Since \(g(u)=f(\e^u)\), the substitution \(t=\e^u\) proves
	\eqref{eq:real-power-representation}.
	
	Taking \(u=0\) in \eqref{eq:g-bilateral-laplace} gives
	\[
	\mu(\mathbb R)=g(0)=f(1)<\infty.
	\]
	Thus \(\mu\) is finite and \eqref{eq:real-power-integrability} holds.
\end{proof}

\subsection{Support of the representing measure}

\begin{lemma}
	\label{lem:monotone-Laplace-support}
	Let \(\mu\) be a finite positive Borel measure on \(\mathbb R\), and assume
	\[
	g(s)=\int_{\mathbb R}\e^{s\alpha}\,d\mu(\alpha)
	\]
	is finite for all \(s\in\mathbb R\).  If \(g\) is nondecreasing on
	\(\mathbb R\), then
	\begin{equation}\label{eq:Laplace-support}
		\supp\mu\subseteq[0,\infty).
	\end{equation}
\end{lemma}

\begin{proof}
Since \(g\) is nondecreasing, \(g(-R)\le g(0)=\mu(\mathbb R)\) for
\(R>0\). If \(\mu((-\infty,0))>0\), there is \(\delta>0\) such that
\(\mu((-\infty,-\delta])>0\), and then
\[
g(-R)\ge \e^{R\delta}\mu((-\infty,-\delta])\longrightarrow\infty.
\]
This contradiction proves \(\supp\mu\subseteq[0,\infty)\).
\end{proof}

\begin{proof}[Proof of Theorem~\ref{thm:PSD-only-preservers}]
	The zero function gives the first alternative.  For necessity, let
	\(f\not\equiv0\) satisfy condition~\textup{(i)}.
	By Proposition~\ref{prop:real-power-representation},
	\[
	f(t)
	=
	\int_{\mathbb R}t^\alpha\,d\mu(\alpha)
	\]
	for a finite positive measure \(\mu\).  By
	Proposition~\ref{prop:automatic-regularity}, \(f\) is nondecreasing, and hence
	\[
	g(s)=f(\e^s)
	\]
	is also nondecreasing.  Lemma~\ref{lem:monotone-Laplace-support} gives
	\[
	\supp\mu\subseteq[0,\infty).
	\]
	This is the representation in \eqref{eq:PSD-preserver-representation}.
	
	Conversely, suppose \eqref{eq:PSD-preserver-representation} holds for a
	positive measure \(\mu\) on \([0,\infty)\).
	Fix \(n\ge1\), \(G=(g_{ij})\in\mathcal G_n^+\), and
	\(x=(x_1,\ldots,x_n)^T\in\mathbb R^n\).  Interchanging the finite sum
	and the integral gives
	\begin{align}
		x^Tf[G]x
		&=
		\sum_{i,j=1}^n x_ix_jf(g_{ij})\notag\\
		&=
		\sum_{i,j=1}^n x_ix_j
		\int_{[0,\infty)}
		g_{ij}^{\alpha}\,d\mu(\alpha)\notag\\
		&=
		\int_{[0,\infty)}
		\sum_{i,j=1}^n
		x_ix_jg_{ij}^{\alpha}\,d\mu(\alpha)\notag\\
		&=
		\int_{[0,\infty)}
		x^TG^{\circ\alpha}x\,d\mu(\alpha).
		\label{eq:power-mixture-quadratic-form}
	\end{align}
	
	For \(\alpha>0\), Proposition~\ref{prop:green-Hadamard-power} gives
	\[
	G^{\circ\alpha}\succeq0.
	\]
	For \(\alpha=0\), \eqref{eq:zero-Hadamard-power} gives
	\[
	G^{\circ0}
	=
	\mathbf1\mathbf1^T\succeq0.
	\]
	Thus the integrand in \eqref{eq:power-mixture-quadratic-form} is
	nonnegative, so
	\[
	x^Tf[G]x\ge0.
	\]
	Since \(x\) was arbitrary,
	\[
	f[G]\succeq0.
	\]
\end{proof}

\begin{remark}
	When \(f\not\equiv0\) in Theorem~\ref{thm:PSD-only-preservers}, the
	representing measure \(\mu\) is unique.  Indeed, set
	\[
	g(s)=f(\e^s).
	\]
	Then \(\mu\) represents the bilateral Laplace transform
	\[
	g(s)=\int_{[0,\infty)}\e^{s\alpha}\,d\mu(\alpha)
	\]
	Uniqueness for bilateral Laplace transforms of positive measures on their
	domain gives the assertion.
\end{remark}

\subsection{A reduced test family}

The proof of the main theorem uses only two families of Green matrices:

\begin{enumerate}
	\item all positive rank-one matrices
	\[
	xx^T;
	\]
	\item all matrices
	\[
	\begin{pmatrix}
		a&a\\
		a&b
	\end{pmatrix},
	\qquad 0<a\le b.
	\]
\end{enumerate}

\begin{corollary}
\label{cor:reduced-test-family}
	Let \(f:(0,\infty)\to\mathbb R\).  The following are equivalent.
	\begin{enumerate}[label=\textup{(\roman*)}]
		\item \(f[G]\succeq0\) for every \(n\ge1\) and
		\(G\in\mathcal G_n^+\);
		
		\item for every \(n\ge1\) and \(x\in(0,\infty)^n\),
		\[
		\bigl(f(x_ix_j)\bigr)_{i,j=1}^n\succeq0,
		\]
		and, for every \(0<a\le b\),
		\[
		\begin{pmatrix}
			f(a)&f(a)\\
			f(a)&f(b)
		\end{pmatrix}
		\succeq0.
		\]
		
		\item either \(f\equiv0\), or
		\[
		f(t)
		=
		\int_{[0,\infty)}
		t^\alpha\,d\mu(\alpha)
		\]
		where \(\mu\) is a nonzero finite positive Borel measure on
		\([0,\infty)\).
	\end{enumerate}
\end{corollary}

\begin{proof}
	That \textup{(i)} implies \textup{(ii)} follows immediately for the
	positive rank-one matrices, while
	\[
	\begin{pmatrix}
		a&a\\
		a&b
	\end{pmatrix}
	\]
	belongs to \(\mathcal G_2^+\).
	
	Under \textup{(ii)}, the rank-one tests show directly that
	\(K(s,t)=f(\e^{s+t})\) is a PSD Hankel kernel, while the tests of order two give
	\[
	f(a)\bigl(f(b)-f(a)\bigr)\ge0.
	\]
	The zero-or-positive argument in Proposition~\ref{prop:automatic-regularity} shows
	that either \(f\equiv0\), or \(f>0\) and \(f\) is nondecreasing.  The
	proofs of Proposition~\ref{prop:real-power-representation} and
	Lemma~\ref{lem:monotone-Laplace-support} now apply and give
	\textup{(iii)}.
	
	Finally, \textup{(iii)}\(\Rightarrow\)\textup{(i)} follows from the
	quadratic-form calculation in the proof of
	Theorem~\ref{thm:PSD-only-preservers}.
\end{proof}

\subsection{Fixed matrix order}\label{sec:fixed}

The all-order representation above does not describe the preserver class
for a single order. We now reduce that problem to a family with two scalar
parameters.

\begin{theorem}\label{thm:fixed}
Fix \(n\ge2\), and let \(f:(0,\infty)\to\mathbb R\). The following are
equivalent.
\begin{enumerate}[label=\textup{(\roman*)}]
\item \(f[G]\succeq0\) for every \(G\in\mathcal G_n^+\).
\item The function \(f\) is nondecreasing, and \(f[xx^T]\succeq0\) for
every \(x\in(0,\infty)^n\).
\item The function \(f\) is nondecreasing, and
\begin{equation}\label{eq:gp}
\bigl(f(tr^{i+j})\bigr)_{i,j=0}^{n-1}\succeq0,
\qquad t>0,\quad r>1.
\end{equation}
\end{enumerate}
Every nonzero function satisfying these conditions is positive and
continuous on \((0,\infty)\).
\end{theorem}

The passage from geometric progressions to arbitrary rank-one inputs uses
a finite-order Hankel criterion. The strict positive-definiteness
implication is the classical criterion quoted by Bodirsky, Kummer, and
Thom~\cite[Lemma~4.1]{BodirskyKummerThom2026} from Karlin's theory. The
following statement includes its semidefinite consequence.

\begin{lemma}\label{lem:finite-hankel}
Let \(n\ge2\) and \(g\in C^{2n-2}(\mathbb R)\). Then
\[
\bigl(g(s_i+s_j)\bigr)_{i,j=1}^n\succeq0
\quad\text{for all }s_1,\ldots,s_n\in\mathbb R
\]
if and only if
\begin{equation}\label{eq:dh}
H_n(g;u):=\bigl(g^{(i+j)}(u)\bigr)_{i,j=0}^{n-1}\succeq0,
\qquad u\in\mathbb R.
\end{equation}
\end{lemma}

\begin{proof}
For necessity, take the nodes \(s_i=u/2+ih\), \(0\le i<n\), where \(h>0\).
Apply the congruence with lower triangular matrix
\begin{equation}\label{eq:diff}
(T_h)_{ik}=
\begin{cases}
(-1)^{i-k}\binom{i}{k}h^{-i},&0\le k\le i<n,\\
0,&k>i.
\end{cases}
\end{equation}
The \((i,j)\) entry of the resulting matrix is
\(h^{-i-j}\Delta_h^{i+j}g(u)\), where
\(\Delta_hg(u)=g(u+h)-g(u)\). As \(h\downarrow0\), these matrices converge
to \(H_n(g;u)\), proving \eqref{eq:dh}.

For sufficiency, put \(p(u)=\sum_{j=0}^{n-1}\e^{ju}\) and
\(v_j=(1,j,\ldots,j^{n-1})^T\). The Vandermonde vectors \(v_j\) are
linearly independent, and
\[
H_n(p;u)=\sum_{j=0}^{n-1}\e^{ju}v_jv_j^T\succ0.
\]
Thus \(H_n(g+\varepsilon p;u)\succ0\) for every \(u\) and
\(\varepsilon>0\). The strict Hankel criterion
\cite[Lemma~4.1]{BodirskyKummerThom2026} states that this condition makes
\(\bigl((g+\varepsilon p)(s_i+s_j)\bigr)\) positive definite whenever the
\(s_i\)'s are distinct. Repeated nodes follow by continuity. Letting
\(\varepsilon\downarrow0\) proves the required semidefinite statement.
\end{proof}

\begin{proof}[Proof of Theorem~\ref{thm:fixed}]
Assume \textup{(i)}. Positive rank-one matrices belong to
\(\mathcal G_n^+\). For \(0<a<b\), apply \textup{(i)} to
\(A_{\min}(a,b,\ldots,b)\). Its output quadratic form on
\((1,-1,0,\ldots,0)^T\) is \(f(b)-f(a)\), so \(f\) is nondecreasing.
This proves \textup{(ii)}, which implies \textup{(iii)} by taking
\(x_i=\sqrt t\,r^i\), \(0\le i<n\).

We prove \textup{(ii)}\(\Rightarrow\)\textup{(i)} before returning to the
geometric-progression tests. Set \(g(u)=f(\e^u)\) and
\(m=\lfloor n/2\rfloor\). Condition \textup{(ii)} says that \(g\) is
nondecreasing and that all its Hankel evaluation matrices on at most \(n\)
points are PSD; shorter lists are obtained by repeating points and taking
principal submatrices. In particular, \(g\ge0\).

We first show that, for every \(h>0\), the function
\(g(u+h)-g(u)\) has PSD Hankel evaluation matrices on at most \(m\) points.
Fix \(r\le m\), points \(s_1,\ldots,s_r\), and coefficients
\(z_1,\ldots,z_r\). Define
\[
H(t)=\sum_{i,j=1}^r z_iz_jg(s_i+s_j+t).
\]
Apply the \(2r\)-point Hankel condition to the points
\(s_i+t/2\), \(s_i+t/2+a\), with coefficients \(z_i,-z_i\), respectively.
For every \(a>0\), this gives
\[
H(t+2a)-2H(t+a)+H(t)\ge0.
\]
Thus \(H\) is midpoint convex. The finite sum defining it is locally
bounded because \(g\) is nonnegative and nondecreasing, so \(H\) is convex.
Moreover, \(g\) has a finite limit \(L\ge0\) at \(-\infty\), and
\[
\lim_{t\to-\infty}H(t)=L\left(\sum_{i=1}^r z_i\right)^2.
\]
A convex function with a finite limit at \(-\infty\) is nondecreasing:
a negative secant slope would give an unbounded lower estimate towards
\(-\infty\). Hence \(H(h)-H(0)\ge0\), which proves the asserted
Hankel positivity of every translation increment.

Now write a Green input as \(G=D_qA_{\min}(r)D_q\), and set
\[
t_i=\tfrac12\log r_i,\qquad
u_i=\log q_i+t_i.
\]
Then \(t_1\le\cdots\le t_n\) and
\[
\log G_{ij}=u_i+u_j-|t_i-t_j|.
\]
Choose a middle index \(k\) so that both sets \(i<k\) and \(i>k\) contain
at most \(m\) indices. Put \(d_i=|t_i-t_k|\) and \(w_i=u_i-d_i\).
If \(i\) and \(j\) lie on opposite sides of \(k\), or one equals \(k\),
then \(\log G_{ij}=w_i+w_j\). If they lie on the same side, then
\[
\log G_{ij}=w_i+w_j+2\min(d_i,d_j).
\]
It follows that \(f[G]\) is the PSD matrix
\(\bigl(g(w_i+w_j)\bigr)_{i,j=1}^n\) plus correction matrices supported on
the two diagonal blocks with indices \(i<k\) and \(i>k\).

For either side \(S\), list its distinct positive distances as
\(0=\lambda_0<\lambda_1<\cdots<\lambda_N\). Let \(D_\ell\) be the diagonal
matrix on \(S\) whose \(i\)-th diagonal entry is \(1\) if
\(d_i\ge\lambda_\ell\), and \(0\) otherwise. Its correction block is
\[
Q_S=\sum_{\ell=1}^N D_\ell M_\ell D_\ell,
\qquad
(M_\ell)_{ij}
=g(w_i+w_j+2\lambda_\ell)
-g(w_i+w_j+2\lambda_{\ell-1}).
\]
Indeed, the \((i,j)\) entry telescopes through the levels up to
\(\min(d_i,d_j)\). Each \(M_\ell\) is a Hankel evaluation matrix of a
translation increment, on the at most \(m\) points
\(w_i+\lambda_{\ell-1}\). The preceding argument makes it PSD.
Thus \(Q_S\succeq0\); an empty side or a side with no positive distance
has zero correction. Both corrections are PSD, proving \(f[G]\succeq0\).

Finally, assume \textup{(iii)}. Its two-by-two principal tests cover every
positive rank-one matrix of order two: unequal positive coordinates are
\(\sqrt t,\sqrt t\,r\) after reordering, and equal coordinates follow from
the one-by-one tests. Together with monotonicity, the zero-or-positive and
continuity arguments of Proposition~\ref{prop:automatic-regularity}
therefore apply. The zero function needs no further argument. Otherwise \(f\) is positive
and continuous, so \(g(u)=f(\e^u)\) is continuous.

Let \(\rho_\varepsilon\) be a nonnegative smooth function supported in
\([-\varepsilon,\varepsilon]\), with integral one, and put
\[
g_\varepsilon(u)=\int_{\mathbb R}
\rho_\varepsilon(v)g(u-v)\,dv.
\]
These functions are smooth and converge locally uniformly to \(g\).
For \(u\in\mathbb R\) and \(h>0\), condition \eqref{eq:gp}, integrated
against \(\rho_\varepsilon\), gives
\[
\bigl(g_\varepsilon(u+(i+j)h)\bigr)_{i,j=0}^{n-1}\succeq0.
\]
Apply the forward-difference congruence \eqref{eq:diff} and let
\(h\downarrow0\). It follows that
\(H_n(g_\varepsilon;u)\succeq0\) for every \(u\). Lemma~\ref{lem:finite-hankel}
then gives arbitrary \(n\)-point Hankel positivity for
\(g_\varepsilon\). Letting \(\varepsilon\downarrow0\) yields that of
\(g\), and therefore \textup{(ii)}. The automatic regularity assertion was
proved in the same argument.
\end{proof}

For order one, the condition is simply \(f\ge0\); monotonicity is not
required.

\begin{corollary}\label{cor:euler}
Fix \(n\ge2\) and \(f\in C^{2n-2}(0,\infty)\). Define the Euler operator
\(\mathcal D=t\,d/dt\), with \(\mathcal D^0f=f\). Then \(f\) preserves PSD
on \(\mathcal G_n^+\) if and only if
\[
f'(t)\ge0\quad\text{and}\quad
\bigl(\mathcal D^{i+j}f(t)\bigr)_{i,j=0}^{n-1}\succeq0,
\qquad t>0.
\]
\end{corollary}

\begin{proof}
For \(g(u)=f(\e^u)\), the chain rule gives
\(g^{(j)}(u)=\mathcal D^jf(\e^u)\).
Apply Lemma~\ref{lem:finite-hankel} and Theorem~\ref{thm:fixed}.
\end{proof}

\begin{corollary}\label{cor:hierarchy}
For every \(n\ge2\), the PSD preserver class on \(\mathcal G_{n+1}^+\) is
a proper subset of that on \(\mathcal G_n^+\).
More precisely, \(f(t)=(1+t)^\alpha\) belongs to the latter class and not
the former whenever \(n-2<\alpha<n-1\).
\end{corollary}

\begin{proof}
The inclusion follows from Theorem~\ref{thm:fixed} by taking leading
principal submatrices in \eqref{eq:gp}.
The FitzGerald--Horn theorem~\cite{FitzGeraldHorn} gives
\(C^{\circ\alpha}\succeq0\) for every \(n\times n\) PSD matrix \(C\) with
nonnegative entries when \(\alpha\ge n-2\). Applied to
\(C=\mathbf1\mathbf1^T+G\), this proves preservation on
\(\mathcal G_n^+\).

For failure in order \(n+1\), choose distinct positive
\(u_1,\ldots,u_{n+1}\) and set \(x_i=\varepsilon u_i\). The binomial
expansion and determinant multilinearity give
\[
\det f[xx^T]
=\varepsilon^{n(n+1)}
\left(\prod_{k=0}^n\binom\alpha k\right)
\prod_{i<j}(u_j-u_i)^2
+O\bigl(\varepsilon^{n(n+1)+2}\bigr).
\]
The first nonzero term uses the distinct powers \(0,1,\ldots,n\); a repeated
power gives two proportional columns and vanishes. For
\(n-2<\alpha<n-1\), the coefficients \(\binom\alpha k\) are positive for
\(0\le k<n\), while \(\binom\alpha n<0\). Thus the determinant is negative
for sufficiently small positive \(\varepsilon\). Since \(xx^T\) is Green,
this proves the strict inclusion.
\end{proof}

\section{Loewner monotonicity}\label{sec:Loewner-fixed-q}

For a fixed vector \(q\in(0,\infty)^n\), set
\begin{equation}\label{eq:fixed-q-family}
\mathcal G_n^+(q)
=\left\{D_qA_{\min}(r)D_q:0<r_1\le\cdots\le r_n\right\}.
\end{equation}
This family is convex, since
\[
(1-s)D_qA_{\min}(r)D_q+sD_qA_{\min}(u)D_q
=D_qA_{\min}((1-s)r+su)D_q,
\qquad 0\le s\le1.
\]
The entire Green class is not convex. For example, the positive rank-one
matrices \(xx^T\) and \(yy^T\), where \(x=(1,1,1)^T\) and
\(y=(1,2,3)^T\), are Green, whereas
\[
xx^T+yy^T=
\begin{pmatrix}2&3&4\\3&5&7\\4&7&10\end{pmatrix}
\]
is not: its entries violate the Green identity
\(a_{12}a_{23}=a_{13}a_{22}\).
Nevertheless, the two comparison problems have the same answer.

\begin{theorem}\label{thm:Loewner-green-fibers}
Let \(f:(0,\infty)\to\mathbb R\). The following are equivalent.
\begin{enumerate}[label=\textup{(\roman*)}]
\item For every \(n\ge1\) and all \(A,B\in\mathcal G_n^+\),
\[
A\succeq B\quad\Longrightarrow\quad f[A]\succeq f[B].
\]
\item The same implication holds whenever \(A\) and \(B\) belong to the
same family \(\mathcal G_n^+(q)\).
\item There are \(c\in\mathbb R\) and a finite positive Borel measure
\(\mu\) on \([0,\infty)\), possibly zero, such that
\begin{equation}\label{eq:Loewner-fiber-representation}
f(t)=c+\int_{[0,\infty)}
\frac{t^{\alpha+1}-1}{\alpha+1}\,d\mu(\alpha),
\qquad t>0,
\end{equation}
and the integral is finite for every \(t>0\).
\end{enumerate}
Every such function belongs to \(C^\infty(0,\infty)\), and
\begin{equation}\label{eq:deriv}
f'(t)=\int_{[0,\infty)}t^\alpha\,d\mu(\alpha).
\end{equation}
\end{theorem}

We first record the inverse-matrix characterization used to compare Green
matrices with different parameter sequences. The relation between Green
matrices and inverse tridiagonal \(M\)-matrices is classical
\cite{Markham1972,McDonaldNabbenNeumannSchneiderTsatsomeros1998,Nabben2001}.
The following proof gives the precise form needed here.

\begin{lemma}\label{lem:inverse}
A symmetric matrix \(G\) is strict Green if and only if \(G^{-1}\) is
positive definite and tridiagonal, with strictly negative entries on its
first subdiagonal.
\end{lemma}

\begin{proof}
For a strict Green matrix,
\[
G^{-1}=D_q^{-1}L^{-T}
\diag(\Delta_1^{-1},\ldots,\Delta_n^{-1})L^{-1}D_q^{-1}.
\]
Here \(L^{-1}\) has diagonal entries \(1\), first subdiagonal entries \(-1\),
and all other entries zero. Thus \(G^{-1}\) is tridiagonal, and its
\((i,i+1)\) entry is \(-1/(q_iq_{i+1}\Delta_{i+1})<0\).

Conversely, let \(K\) be positive definite and symmetric tridiagonal, and
write \(K_{i,i+1}=-b_i<0\). Let \(\theta_k\) be the determinant of its
leading \(k\times k\) principal submatrix, with \(\theta_0=1\), and let
\(\phi_k\) be the determinant of its principal submatrix on indices
\(k,\ldots,n\), with \(\phi_{n+1}=1\). These determinants are positive.
For \(i\le j\), the cofactor formula and expansion along the intervening
tridiagonal rows give
\[
(K^{-1})_{ij}
=\frac{\theta_{i-1}\phi_{j+1}}{\theta_n}
\prod_{k=i}^{j-1}b_k.
\]
The empty product is \(1\). With \(P_i=\prod_{k<i}b_k\), this becomes
\[
(K^{-1})_{ij}
=\left(\frac{\theta_{i-1}}{P_i}\right)
\left(\frac{P_j\phi_{j+1}}{\theta_n}\right).
\]
Both factors are positive, so \(K^{-1}\) is a positive-parameter Green
matrix. Since it is positive definite, its ratio sequence is strictly
increasing by Corollary~\ref{cor:green-positivity-criterion}.
\end{proof}

\begin{proof}[Proof of Theorem~\ref{thm:Loewner-green-fibers}]
The implication \textup{(i)}\(\Rightarrow\)\textup{(ii)} is immediate.
Assume \textup{(ii)}. Scalar tests make \(f\) nondecreasing. For
\(0<s\le t\) and \(h>0\), the matrices
\(A_{\min}(s+h,t+h)\) and \(A_{\min}(s,t)\) differ by
\(h\mathbf1\mathbf1^T\succeq0\). The difference of their entrywise images
is a Min matrix with parameters
\[
f(s+h)-f(s),\qquad f(t+h)-f(t).
\]
Its positive semidefiniteness gives
\[
0\le f(s+h)-f(s)\le f(t+h)-f(t).
\]
Taking \(t=s+h\) proves midpoint convexity. Local boundedness follows from
monotonicity, so \(f\) is convex and continuous on \((0,\infty)\).
Its right derivative \(g=f'_+\) therefore exists and is finite everywhere.

Let \(G=D_qA_{\min}(r)D_q\in\mathcal G_n^+(q)\). For \(\varepsilon>0\),
\[
G+\varepsilon qq^T
=D_qA_{\min}(r+\varepsilon\mathbf1)D_q
\in\mathcal G_n^+(q).
\]
Condition \textup{(ii)} gives
\[
\frac{f[G+\varepsilon qq^T]-f[G]}{\varepsilon}\succeq0.
\]
Passing to the entrywise limit as \(\varepsilon\downarrow0\) yields
\(D_qg[G]D_q\succeq0\), hence \(g[G]\succeq0\).
Theorem~\ref{thm:PSD-only-preservers} now gives a finite positive measure
\(\mu\) on \([0,\infty)\), with the zero measure allowed, such that
\[
g(t)=\int_{[0,\infty)}t^\alpha\,d\mu(\alpha),
\qquad t>0.
\]
In particular, \(g\) is continuous. For a convex function, its left
derivative at \(t\) is the limit of its right derivatives at points
increasing to \(t\). Thus \(f'_-(t)=g(t)=f'_+(t)\), and \(f\in C^1\)
with \(f'=g\). Integrating from \(1\) to \(t\) proves \textup{(iii)}.

Conversely, assume \textup{(iii)}. On any compact subinterval of
\((0,\infty)\), differentiation under the integral is justified by choosing
a larger argument \(T>1\): exponential growth in \(\alpha\) at \(T\)
dominates every fixed power of \(\alpha\) arising from differentiation.
Consequently \(f\in C^\infty(0,\infty)\) and \eqref{eq:deriv} holds.
Theorem~\ref{thm:PSD-only-preservers} gives
\(f'[C]\succeq0\) for every \(C\in\mathcal G_n^+\).

First suppose that \(A\succeq B\) and both matrices are strict Green.
For \(0\le s\le1\), put
\[
K_s=(1-s)B^{-1}+sA^{-1},\qquad C_s=K_s^{-1}.
\]
The matrix \(K_s\) is positive definite and tridiagonal, with negative first
subdiagonal entries, so Lemma~\ref{lem:inverse} makes \(C_s\) strict Green.
Moreover, inverse order gives \(B^{-1}-A^{-1}\succeq0\), and hence
\[
C_s'=C_s(B^{-1}-A^{-1})C_s\succeq0.
\]
The Schur product theorem therefore gives
\[
\frac{d}{ds}f[C_s]=f'[C_s]\circ C_s'\succeq0.
\]
Since \(C_0=B\) and \(C_1=A\), integration proves
\(f[A]\succeq f[B]\).

For general \(A,B\in\mathcal G_n^+\), set
\[
E_\varepsilon=(\e^{-\varepsilon|i-j|})_{i,j=1}^n,
\qquad\varepsilon>0.
\]
This is a strict Green matrix: take \(p_i=\e^{\varepsilon i}\) and
\(q_i=\e^{-\varepsilon i}\). If \(G=G(p,q)\in\mathcal G_n^+\), then
\[
G\circ E_\varepsilon
=G\bigl((p_i\e^{\varepsilon i})_i,
        (q_i\e^{-\varepsilon i})_i\bigr)
\in\mathcal G_n^{\mathrm{str}},
\]
since the new ratios are \((p_i/q_i)\e^{2\varepsilon i}\).
As \(E_\varepsilon\succeq0\), the Schur product theorem also gives
\[
A\circ E_\varepsilon\succeq B\circ E_\varepsilon.
\]
The strict case applies to these two matrices. Letting
\(\varepsilon\downarrow0\) and using the continuity of \(f\) proves
\textup{(i)}.
\end{proof}

\begin{corollary}\label{cor:fixed-order}
Fix \(n\ge2\), and let \(f:(0,\infty)\to\mathbb R\). The following are
equivalent.
\begin{enumerate}[label=\textup{(\roman*)}]
\item For all \(A,B\in\mathcal G_n^+\),
\(A\succeq B\) implies \(f[A]\succeq f[B]\).
\item The same implication holds on each fixed-\(q\) family
\(\mathcal G_n^+(q)\).
\item The function \(f\) belongs to \(C^1(0,\infty)\), its derivative is
nondecreasing, and
\[
\bigl(f'(tr^{i+j})\bigr)_{i,j=0}^{n-1}\succeq0,
\qquad t>0,\quad r>1.
\]
\end{enumerate}
\end{corollary}

\begin{proof}
The implication \textup{(i)}\(\Rightarrow\)\textup{(ii)} is immediate.
Under \textup{(ii)}, scalar monotonicity follows by comparing constant
rank-one Green matrices. To obtain convexity, compare
\[
A_{\min}(s,t,t,\ldots,t)\quad\text{and}\quad
A_{\min}(s+h,t+h,t+h,\ldots,t+h),
\]
where \(0<s\le t\) and \(h>0\).
Their difference is \(h\mathbf1\mathbf1^T\succeq0\); the leading
\(2\times2\) principal submatrix of the output difference gives
\[
0\le f(s+h)-f(s)\le f(t+h)-f(t).
\]
As in Theorem~\ref{thm:Loewner-green-fibers}, this makes \(f\) convex.
Set \(g=f'_+\). The difference quotients along
\(G+\varepsilon qq^T\) show that \(g[G]\succeq0\) for every
\(G\in\mathcal G_n^+\), with the order fixed at \(n\).
Theorem~\ref{thm:fixed} makes \(g\) continuous, so the left- and
right-derivative argument in Theorem~\ref{thm:Loewner-green-fibers} gives
\(f\in C^1\) and \(f'=g\). Applying Theorem~\ref{thm:fixed} to \(g\)
proves \textup{(iii)}.

Conversely, \textup{(iii)} and Theorem~\ref{thm:fixed} give
\(f'[G]\succeq0\) for every \(G\in\mathcal G_n^+\).
The inverse-matrix path and the common Schur regularization in the proof of
Theorem~\ref{thm:Loewner-green-fibers} preserve the order \(n\) throughout.
Their Schur-product and limiting arguments therefore prove \textup{(i)}.
\end{proof}

In particular, \(t\mapsto t^\alpha\), for real \(\alpha\), is entrywise
Loewner monotone on the Green class in every order exactly when
\[
\alpha=0\quad\text{or}\quad\alpha\ge1.
\]
Indeed, negative exponents fail scalar monotonicity, and exponents in
\((0,1)\) fail the midpoint-convexity condition above. The remaining
exponents satisfy \eqref{eq:deriv}.

\begin{corollary}\label{cor:strict-diff}
Let \(f\) satisfy Theorem~\ref{thm:Loewner-green-fibers}, and suppose that
\(f\) is not affine. If \(A\succeq B\) are strict Green matrices, then
\[
\ker\bigl(f[A]-f[B]\bigr)
=\operatorname{span}\{e_i:A_{ii}=B_{ii}\},
\]
where \(e_1,\ldots,e_n\) are the standard coordinate vectors. In particular,
the rank of \(f[A]-f[B]\) is the number of indices with \(A_{ii}>B_{ii}\).
\end{corollary}

\begin{proof}
Use the path \(C_s\) from the proof of
Theorem~\ref{thm:Loewner-green-fibers}. Nonaffineness means that the measure
in \eqref{eq:deriv} has positive mass on \((0,\infty)\).
For every such exponent, \(C_s^{\circ\alpha}\) is positive definite by
Proposition~\ref{prop:green-Hadamard-power}; integrating quadratic forms
therefore shows that \(f'[C_s]\) is positive definite. By continuity and
compactness, there is \(m>0\) such that
\(f'[C_s]\succeq mI\) for all \(0\le s\le1\).
Write \(H=A-B\) and \(D=f[A]-f[B]\). Since \(C_s'\succeq0\), the Schur
product theorem and integration give
\[
D=\int_0^1 f'[C_s]\circ C_s'\,ds
\succeq m\diag(H_{11},\ldots,H_{nn}).
\]
If \(H_{ii}=0\), then \(D_{ii}=f(A_{ii})-f(B_{ii})=0\). As \(D\succeq0\),
its \(i\)-th row and column vanish. Conversely, if \(x\in\ker D\), the
displayed lower bound forces \(x_i=0\) whenever \(H_{ii}>0\).
These two observations prove the kernel formula.
\end{proof}

The strict-input hypothesis is necessary for this rank formula: for
\(A=2\mathbf1\mathbf1^T\), \(B=\mathbf1\mathbf1^T\), and \(f(t)=t^2\), all
diagonal entries increase, but \(f[A]-f[B]=3\mathbf1\mathbf1^T\) has rank one.

\section{Rigidity of totally nonnegative output}\label{sec:TN-output}

For positive-parameter Green inputs, positive semidefiniteness is equivalent
to total nonnegativity. The corresponding entrywise output requirements are
different: totally nonnegative output forces a single power, and it is
enough to test strict matrices of order three.

\begin{theorem}\label{thm:TN-output-preservers}
Let \(f:(0,\infty)\to\mathbb R\). The following are equivalent.
\begin{enumerate}[label=\textup{(\roman*)}]
\item For every \(n\ge1\) and \(G\in\mathcal G_n^+\), the matrix
\(f[G]\) is totally nonnegative.
\item For every \(G\in\mathcal G_3^{\mathrm{str}}\), the matrix
\(f[G]\) belongs to \(\TN_2\).
\item Either \(f\equiv0\), or \(f(t)=ct^\alpha\) for some \(c>0\) and
\(\alpha\ge0\).
\end{enumerate}
\end{theorem}

Belton, Guillot, Khare, and Putinar
\cite{BeltonGuillotKharePutinar2023} obtained the same function class for
entrywise maps preserving \(\TN_2\) on symmetric \(\TN_2\) matrices of each fixed order at least three. Their rank-one tests force a multiplicative functional
equation. Here the tests are restricted to strict Green matrices of order
three; the following argument works within that smaller class.

\begin{proof}
Condition \textup{(i)} implies \textup{(ii)}. To prove
\textup{(ii)}\(\Rightarrow\)\textup{(iii)}, fix \(a,b,d>0\) and consider
\[
G(a,b,d)=
\begin{pmatrix}
2a^2/d&a&ab/d\\
a&d&b\\
ab/d&b&2b^2/d
\end{pmatrix}.
\]
Its Green parameters \(p=(a,d,2b)\) and \(q=(2a/d,1,b/d)\) have ratios
\(d/2<d<2d\). Thus \(G(a,b,d)\) is strict Green.
The minor of \(f[G(a,b,d)]\) in rows \(1,2\) and columns \(2,3\) gives
\[
f(a)f(b)\ge f(ab/d)f(d).
\]
Replacing \((a,b,d)\) by \((ab/d,d,b)\) gives the reverse inequality.
Consequently
\begin{equation}\label{eq:mult}
f(a)f(b)=f(ab/d)f(d),\qquad a,b,d>0.
\end{equation}

The one-by-one minors give \(f\ge0\). If \(f(1)=0\), taking
\(a=b=u\) and \(d=1\) in \eqref{eq:mult} gives \(f(u)^2=0\), so
\(f\equiv0\). Otherwise \(f(1)>0\), and the choice \(a=b=1,d=t\) gives
\[
f(t)f(1/t)=f(1)^2>0.
\]
Thus \(f\) is positive everywhere, and \eqref{eq:mult} with \(d=1\) says
that \(F=f/f(1)\) is multiplicative.
For \(0<a<b<c\), the strict Min matrix \(A_{\min}(a,b,c)\) gives
\[
f(a)(f(b)-f(a))\ge0
\]
from its leading principal minor of order two. Hence \(f\) is nondecreasing.
The function \(h(s)=\log F(\e^s)\) is therefore additive and nondecreasing,
so \(h(s)=\alpha s\) for some \(\alpha\ge0\). This proves
\textup{(iii)}.

For \(\alpha>0\), Proposition~\ref{prop:green-Hadamard-power} gives
\(cG^{\circ\alpha}\in\mathcal G_n^+\), which is TN by
Remark~\ref{rem:green-TN-criterion}. For \(\alpha=0\), the output is
\(c\mathbf1\mathbf1^T\), also TN. The zero function also gives TN output,
so \textup{(iii)} implies \textup{(i)}.
\end{proof}

Order three cannot be replaced by order two. Indeed, \(f(t)=\e^t\) is not
a single power, but if
\(A=\left(\begin{smallmatrix}a&b\\b&c\end{smallmatrix}\right)\succeq0\)
has positive entries, then \(a+c\ge2\sqrt{ac}\ge2b\), whence
\(\det f[A]=\e^{a+c}-\e^{2b}\ge0\). Thus \(f[A]\) is TN.

\begin{remark}\label{rem:PSD-not-TN}
Among the nonzero PSD preservers of
Theorem~\ref{thm:PSD-only-preservers}, one fixed strict Green matrix detects
all the non-power functions. Namely, take
\[
G_*=\begin{pmatrix}8&2&4\\2&1&2\\4&2&8\end{pmatrix}.
\]
For \(f(t)=\int t^\alpha\,d\mu(\alpha)\), the minor in rows \(1,2\) and
columns \(2,3\) is
\[
f(2)^2-f(1)f(4)
=-\frac12\iint(2^\alpha-2^\beta)^2\,d\mu(\alpha)\,d\mu(\beta).
\]
This is negative unless \(\mu\) is concentrated at one point. For example,
\(f(t)=1+t\) gives the value \(-1\), even though \(f\) preserves PSD.
\end{remark}

\section{Strict Green input and positive definite output}\label{sec:strict-PD}

We retain the PSD-output condition but now require the image of every strict
Green input to be positive definite.  The image is still not required to
have Green structure.

\begin{theorem}
	\label{thm:strict-to-PD}
	Let \(f:(0,\infty)\to\mathbb R\) preserve PSD on the Green class.  The
	following are equivalent.
	\begin{enumerate}[label=\textup{(\roman*)}]
		\item for every \(n\ge2\) and \(G\in\mathcal G_n^{\mathrm{str}}\),
		\[
		f[G]\succ0;
		\]
		
		\item \(f\) admits the representation
		\begin{equation}\label{eq:strict-PD-representation}
			f(t)
			=
			\int_{[0,\infty)}
			t^\alpha\,d\mu(\alpha),
			\qquad t>0,
		\end{equation}
		where \(\mu\) is a nonzero finite positive Borel measure on
		\([0,\infty)\) satisfying
		\begin{equation}\label{eq:strict-PD-positive-mass}
			\mu((0,\infty))>0.
		\end{equation}
	\end{enumerate}
\end{theorem}

\begin{proof}
	Assume \textup{(ii)}, and fix
	\[
	G\in\mathcal G_n^{\mathrm{str}}
	\]
	and a nonzero vector \(x\in\mathbb R^n\).  For every \(\alpha>0\),
	Proposition~\ref{prop:green-Hadamard-power} gives
	\[
	G^{\circ\alpha}\in\mathcal G_n^{\mathrm{str}},
	\]
	and hence
	\[
	G^{\circ\alpha}\succ0.
	\]
	Therefore
	\begin{equation}\label{eq:strict-power-quadratic-positive}
		x^TG^{\circ\alpha}x>0,
		\qquad \alpha>0.
	\end{equation}
	
	Using \eqref{eq:strict-PD-representation} and interchanging the finite sum
	with the integral,
	\begin{align}
		x^Tf[G]x
		&=
		\int_{[0,\infty)}
		x^TG^{\circ\alpha}x\,d\mu(\alpha)\notag\\
		&\ge
		\int_{(0,\infty)}
		x^TG^{\circ\alpha}x\,d\mu(\alpha).
		\label{eq:strict-PD-integral}
	\end{align}
	By \eqref{eq:strict-power-quadratic-positive} and
	\eqref{eq:strict-PD-positive-mass}, the right-hand side is strictly
	positive.  Hence
	\[
	x^Tf[G]x>0.
	\]
	Since the nonzero vector \(x\) was arbitrary,
	\[
	f[G]\succ0.
	\]
	
	Conversely, assume \textup{(i)}.  Since \(f\) preserves PSD on the Green
	class, Theorem~\ref{thm:PSD-only-preservers} says that either
	\(f\equiv0\), or \(f\) has a nonnegative-real-power representation.
	
	The zero function is excluded because
	\[
	f[G]=0
	\]
	is never positive definite.
	
	Hence \eqref{eq:strict-PD-representation} holds for a nonzero finite
	positive Borel measure \(\mu\) on \([0,\infty)\).  If
	\[
	\mu((0,\infty))=0,
	\]
	then \(\mu\) is concentrated at \(0\), so \(f(t)=c\) for some \(c>0\):
	\[
	f(t)=c,
	\qquad t>0.
	\]
	For every \(n\ge2\) and every strict Green matrix \(G\),
	\[
	f[G]=c\mathbf1\mathbf1^T
	\]
	has rank one and is not positive definite, a contradiction.  Therefore
	\[
	\mu((0,\infty))>0.
	\]
\end{proof}

\section{Entrywise maps preserving Green structure}
\label{sec:structure-preservers}

We now require the entrywise image to retain Green structure.  We first give
an intrinsic characterization and then treat the positive, positive
semidefinite, and strict Green classes.

\subsection{An intrinsic characterization}

We use the following characterization to convert preservation of Green
structure into a functional equation.

\begin{proposition}
	\label{prop:green-intrinsic}
	Let \(A=(a_{ij})\) be a symmetric \(n\times n\) matrix with positive
	entries.  The following are equivalent.
	\begin{enumerate}[label=\textup{(\roman*)}]
		\item \(A\) is a Green matrix;
		\item for all \(1\le i\le j\le k\le n\),
		\begin{equation}\label{eq:green-multiplicative-identity}
			a_{ij}a_{jk}=a_{ik}a_{jj}.
		\end{equation}
	\end{enumerate}
\end{proposition}

\begin{proof}
	Suppose \(A=G(p,q)\).  If \(i\le j\le k\), then
	\[
	a_{ij}=p_iq_j,\qquad
	a_{jk}=p_jq_k,\qquad
	a_{ik}=p_iq_k,\qquad
	a_{jj}=p_jq_j.
	\]
	Therefore
	\[
	a_{ij}a_{jk}
	=
	p_iq_jp_jq_k
	=
	p_iq_kp_jq_j
	=
	a_{ik}a_{jj},
	\]
	so \eqref{eq:green-multiplicative-identity} holds.
	
	Conversely, suppose \(A\) satisfies
	\eqref{eq:green-multiplicative-identity}.  Define
	\begin{equation}\label{eq:intrinsic-q}
		q_j
		=
		\frac{a_{1j}}{\sqrt{a_{11}}},
		\qquad j=1,\ldots,n,
	\end{equation}
	and
	\begin{equation}\label{eq:intrinsic-p}
		p_i
		=
		\frac{a_{ii}}{q_i},
		\qquad i=1,\ldots,n.
	\end{equation}
	All entries of \(A\) are positive, so \(p_i,q_i>0\).
	
	Fix \(1\le i\le j\le n\).  Applying
	\eqref{eq:green-multiplicative-identity} to \((1,i,j)\) gives
	\begin{equation}\label{eq:intrinsic-identity-step}
		a_{1i}a_{ij}=a_{1j}a_{ii}.
	\end{equation}
	By \eqref{eq:intrinsic-q}--\eqref{eq:intrinsic-identity-step},
	\begin{align*}
		p_iq_j
		&=
		\frac{a_{ii}}{q_i}q_j\\
		&=
		a_{ii}\frac{a_{1j}}{a_{1i}}\\
		&=
		a_{ij}.
	\end{align*}
	Thus \(a_{ij}=p_iq_j\) when \(i\le j\).  By symmetry of \(A\),
	\[
	a_{ij}
	=
	p_{\min(i,j)}q_{\max(i,j)}
	\]
	for all \(i,j\), and hence \(A=G(p,q)\).
\end{proof}

\begin{remark}
	Identity~\eqref{eq:green-multiplicative-identity} is equivalent to
	\[
	a_{ik}
	=
	\frac{a_{ij}a_{jk}}{a_{jj}},
	\qquad i\le j\le k.
	\]
	Therefore, an entrywise preserver of Green structure must preserve this
	multiplicative identity.
\end{remark}

\subsection{Preservers of positive-entry Green matrices}

Let \(\mathcal G_n\) denote the positive-entry Green matrices of order
\(n\), with no monotonicity requirement on their ratio sequences.

We first consider Green matrices with positive entries, without imposing
positive semidefiniteness or monotonicity conditions.

\begin{theorem}
	\label{thm:positive-green-structure}
	Let
	\[
	f:(0,\infty)\longrightarrow(0,\infty).
	\]
	The following are equivalent.
	\begin{enumerate}[label=\textup{(\roman*)}]
		\item for every \(G\in\mathcal G_3\),
		\[
		f[G]\in\mathcal G_3;
		\]
		
		\item for every \(n\ge1\) and \(G\in\mathcal G_n\),
		\[
		f[G]\in\mathcal G_n;
		\]
		
		\item the normalized function
		\begin{equation}\label{eq:normalized-structure-function}
			F(t)=\frac{f(t)}{f(1)}
		\end{equation}
		is positive and multiplicative:
		\begin{equation}\label{eq:normalized-multiplicativity}
			F(ab)=F(a)F(b),
			\qquad a,b>0;
		\end{equation}
		
		\item there are an additive function
		\[
		\varphi:\mathbb R\to\mathbb R
		\]
		and a constant \(c>0\) such that
		\begin{equation}\label{eq:general-structure-form}
			f(t)
			=
			c\exp\bigl(\varphi(\log t)\bigr),
			\qquad t>0.
		\end{equation}
	\end{enumerate}
\end{theorem}

\begin{proof}
	Clearly, \textup{(ii)} implies \textup{(i)}.
	
	Assume \textup{(i)}.  Fix \(a,b>0\), choose
	\[
	p_1=a,\qquad
	p_2=q_2=1,\qquad
	q_3=b,
	\]
	and choose arbitrary \(q_1,p_3>0\).  The resulting positive Green matrix
	\(G\) of order three satisfies
	\begin{equation}\label{eq:structure-test-entries}
		g_{12}=a,\qquad
		g_{23}=b,\qquad
		g_{22}=1,\qquad
		g_{13}=ab.
	\end{equation}
	By \textup{(i)}, \(f[G]\) is again a Green matrix.  Applying
	Proposition~\ref{prop:green-intrinsic} and
	\eqref{eq:green-multiplicative-identity} to the triple \((1,2,3)\) gives
	\begin{equation}\label{eq:structure-functional-equation}
		f(a)f(b)=f(ab)f(1),
		\qquad a,b>0.
	\end{equation}
	Dividing by \(f(1)^2\) proves
	\eqref{eq:normalized-multiplicativity}; hence \textup{(i)} implies
	\textup{(iii)}.
	
	Now assume \textup{(iii)} and fix \(G(p,q)\in\mathcal G_n\).  By
	multiplicativity of \(F\),
	\begin{align}
		f\bigl(p_{\min(i,j)}q_{\max(i,j)}\bigr)
		&=
		f(1)
		F\bigl(p_{\min(i,j)}q_{\max(i,j)}\bigr)\notag\\
		&=
		f(1)
		F\bigl(p_{\min(i,j)}\bigr)
		F\bigl(q_{\max(i,j)}\bigr).
		\label{eq:structure-image-factorization}
	\end{align}
	Define
	\begin{equation}\label{eq:structure-new-parameters}
		P_i=f(1)F(p_i),
		\qquad
		Q_i=F(q_i).
	\end{equation}
	Since \(F>0\), all \(P_i,Q_i\) are positive.  By
	\eqref{eq:structure-image-factorization},
	\[
	f[G(p,q)]
	=
	G(P,Q).
	\]
	Thus \textup{(iii)} implies \textup{(ii)}.
	
	Finally, we prove equivalence of \textup{(iii)} and \textup{(iv)}.  If
	\(F\) is positive and multiplicative, define
	\[
	\varphi(x)=\log F(\e^x).
	\]
	Then
	\[
	\varphi(x+y)
	=
	\log F(\e^{x+y})
	=
	\log\bigl(F(\e^x)F(\e^y)\bigr)
	=
	\varphi(x)+\varphi(y).
	\]
	Thus \(\varphi\) is additive and
	\[
	F(t)=\exp\bigl(\varphi(\log t)\bigr).
	\]
	Taking \(c=f(1)\) gives \eqref{eq:general-structure-form}.
	
	Conversely, if \(f\) has the form \eqref{eq:general-structure-form},
	additivity of \(\varphi\) immediately gives
	\eqref{eq:normalized-multiplicativity}.
\end{proof}

\begin{remark}
	Theorem~\ref{thm:positive-green-structure} imposes neither positive
	semidefiniteness nor monotonicity, so the additive function \(\varphi\)
	may be discontinuous.  If \(f\) is additionally continuous, measurable,
	locally bounded, or monotone on a nondegenerate interval, then
	\[
	\varphi(x)=\alpha x
	\]
	for some \(\alpha\in\mathbb R\), and hence
	\[
	f(t)=ct^\alpha.
	\]
\end{remark}

\subsection{Preservers of positive semidefinite Green structure}

\begin{theorem}
	\label{thm:PSD-green-structure}
	Let
	\[
	f:(0,\infty)\longrightarrow(0,\infty).
	\]
	The following are equivalent.
	\begin{enumerate}[label=\textup{(\roman*)}]
		\item for every \(G\in\mathcal G_3^+\),
		\[
		f[G]\in\mathcal G_3^+;
		\]
		
		\item for every \(n\ge1\) and \(G\in\mathcal G_n^+\),
		\[
		f[G]\in\mathcal G_n^+;
		\]
		
		\item there are \(c>0\) and \(\alpha\ge0\) such that
		\begin{equation}\label{eq:PSD-green-structure-form}
			f(t)=ct^\alpha,
			\qquad t>0.
		\end{equation}
	\end{enumerate}
\end{theorem}

\begin{proof}
	Clearly, \textup{(ii)} implies \textup{(i)}.
	
	Assume \textup{(i)}.  Fix \(a,b>0\), choose
	\[
	p_1=a,\qquad
	p_2=q_2=1,\qquad
	q_3=b,
	\]
	and then choose
	\[
	q_1>a,\qquad p_3>b.
	\]
	The resulting ratio sequence satisfies
	\[
	\frac{a}{q_1}<1<\frac{p_3}{b},
	\]
	so the resulting matrix is strict Green, and in particular belongs to
	\(\mathcal G_3^+\).  The same test as in
	Theorem~\ref{thm:positive-green-structure} gives
	\begin{equation}\label{eq:PSD-structure-functional-equation}
		f(a)f(b)=f(ab)f(1).
	\end{equation}
	Therefore
	\[
	F(t)=\frac{f(t)}{f(1)}
	\]
	is positive and multiplicative.
	
	We next prove that \(f\) is nondecreasing.  Fix \(0<a\le b\) and choose
	\(u\ge b\).  The matrix
	\begin{equation}\label{eq:PSD-structure-monotonicity-test}
		M(a,b,u)
		=
		\begin{pmatrix}
			a&a&a\\
			a&b&b\\
			a&b&u
		\end{pmatrix}
	\end{equation}
	belongs to \(\mathcal G_3^+\).  Hence
	\[
	f[M(a,b,u)]
	\in\mathcal G_3^+
	\]
	and is in particular positive semidefinite.  Its leading principal minor
	of order two is nonnegative, so
	\[
	f(a)\bigl(f(b)-f(a)\bigr)\ge0.
	\]
	Since \(f(a)>0\),
	\[
	f(a)\le f(b).
	\]
	Thus both \(f\) and \(F\) are nondecreasing.
	
	Define
	\[
	\varphi(x)=\log F(\e^x).
	\]
	Then \(\varphi\) is additive and nondecreasing, so there is
	\(\alpha\ge0\) such that
	\[
	\varphi(x)=\alpha x.
	\]
	Consequently,
	\[
	f(t)=f(1)t^\alpha.
	\]
	Taking \(c=f(1)>0\) gives \eqref{eq:PSD-green-structure-form}.
	
	Conversely, if
	\[
	f(t)=ct^\alpha,
	\qquad c>0,\quad\alpha>0,
	\]
	then Proposition~\ref{prop:green-Hadamard-power} gives
	\[
	f[G]
	=
	cG^{\circ\alpha}\in\mathcal G_n^+
	\]
	for every \(G\in\mathcal G_n^+\).  If \(\alpha=0\), then
	\[
	f[G]=c\mathbf1\mathbf1^T,
	\]
	which is a rank-one positive semidefinite Green matrix.  Thus
	\textup{(iii)} implies \textup{(ii)}.
\end{proof}

\subsection{Preservers of strict Green structure}

\begin{theorem}
	\label{thm:strict-green-structure}
	Let
	\[
	f:(0,\infty)\longrightarrow(0,\infty).
	\]
	The following are equivalent.
	\begin{enumerate}[label=\textup{(\roman*)}]
		\item for every \(G\in\mathcal G_3^{\mathrm{str}}\),
		\[
		f[G]\in\mathcal G_3^{\mathrm{str}};
		\]
		
		\item for every \(n\ge1\) and
		\(G\in\mathcal G_n^{\mathrm{str}}\),
		\[
		f[G]\in\mathcal G_n^{\mathrm{str}};
		\]
		
		\item there are \(c>0\) and \(\alpha>0\) such that
		\begin{equation}\label{eq:strict-green-structure-form}
			f(t)=ct^\alpha,
			\qquad t>0.
		\end{equation}
	\end{enumerate}
\end{theorem}

\begin{proof}
	Clearly, \textup{(ii)} implies \textup{(i)}.
	
	Assume \textup{(i)}.  The test matrix used to obtain
	\eqref{eq:PSD-structure-functional-equation} is strict. Hence
	assumption~\textup{(i)} gives
	\[
	f(a)f(b)=f(ab)f(1),
	\qquad a,b>0.
	\]
	Thus \(F=f/f(1)\) is positive and multiplicative.
	
	Fix \(0<a<b\) and choose \(u>b\).  The matrix
	\[
	M(a,b,u)
	=
	\begin{pmatrix}
		a&a&a\\
		a&b&b\\
		a&b&u
	\end{pmatrix}
	\]
	belongs to \(\mathcal G_3^{\mathrm{str}}\).  Hence
	\(f[M(a,b,u)]\) is again strict Green and therefore positive definite.
	Its leading principal minor of order two is positive:
	\[
	f(a)\bigl(f(b)-f(a)\bigr)>0.
	\]
	Thus
	\[
	f(a)<f(b),
	\]
	so \(F\) is strictly increasing.
	
	Consequently,
	\[
	\varphi(x)=\log F(\e^x)
	\]
	is additive and strictly increasing.  Therefore
	\[
	\varphi(x)=\alpha x
	\]
	for some \(\alpha>0\).  Hence
	\[
	f(t)=f(1)t^\alpha.
	\]
	
	Conversely, if
	\[
	f(t)=ct^\alpha,
	\qquad c>0,\quad\alpha>0,
	\]
	then, for every strict Green matrix \(G(p,q)\),
	\[
	f[G(p,q)]
	=
	G(cp^\alpha,q^\alpha).
	\]
	The new ratio sequence is
	\[
	c\left(\frac{p_i}{q_i}\right)^\alpha,
	\]
	which is still strictly increasing.  Thus \(f[G]\) is strict Green.
\end{proof}

\subsection{Optimality of order three}

Every symmetric two-by-two matrix with positive entries,
	\[
	A=
	\begin{pmatrix}
		a&b\\
		b&c
	\end{pmatrix},
	\qquad a,b,c>0,
	\]
	is a Green matrix.  Indeed, take
	\[
	p_1=q_1=\sqrt a,\qquad q_2=b/\sqrt a,
	\qquad p_2=c\sqrt a/b.
	\]
	The two ratios are \(1\) and \(ac/b^2\).  Hence \(A\) is strict Green if
	and only if
	\[
	ac>b^2,
	\]
	which is precisely the condition \(A\succ0\).  We use this elementary
	observation only to establish sharpness of the order-three tests.

\begin{proposition}
	\label{prop:order-three-optimal}
	The function
	\[
	f(t)=\e^t
	\]
	is not of the form \(ct^\alpha\), but it maps every positive-entry PSD
	Green matrix of order two to a PSD Green matrix, and every strict Green
	matrix of order two to a strict Green matrix.  Thus order three in
	Theorems~\ref{thm:PSD-green-structure} and
	\ref{thm:strict-green-structure} cannot be reduced to order two.
\end{proposition}

\begin{proof}
	Let
	\[
	A=
	\begin{pmatrix}
		a&b\\
		b&c
	\end{pmatrix}
	\succeq0,
	\qquad a,b,c>0.
	\]
	Then
	\[
	ac\ge b^2.
	\]
	By the arithmetic--geometric mean inequality,
	\[
	a+c\ge2\sqrt{ac}\ge2b.
	\]
	Therefore
	\begin{align}
		\det f[A]
		&=
		\det
		\begin{pmatrix}
			\e^a&\e^b\\
			\e^b&\e^c
		\end{pmatrix}\notag\\
		&=
		\e^{a+c}-\e^{2b}\ge0.
		\label{eq:exp-two-by-two}
	\end{align}
	Since the diagonal entries of \(f[A]\) are positive,
	\[
	f[A]\succeq0.
	\]
	The preceding observation shows that \(f[A]\) is still Green.
	
	If \(A\succ0\), then
	\[
	ac>b^2,
	\]
	and hence
	\[
	a+c\ge2\sqrt{ac}>2b.
	\]
	Thus the inequality in \eqref{eq:exp-two-by-two} is strict, so
	\[
	f[A]\succ0.
	\]
	By the same observation, \(f[A]\) is strict Green.
	
	Finally, \(\e^t\) cannot agree with \(ct^\alpha\) on all of
	\((0,\infty)\), so tests of order two cannot force the power form.
\end{proof}

\section{Zero entries}\label{sec:zero}

For each \(n\), let \(\mathcal C_n\) be the entrywise closure of
\(\mathcal G_n^+\) in the space of symmetric matrices. Thus an element of
\(\mathcal C_n\) is a limit of positive-parameter PSD Green matrices of
order \(n\). Its entries are nonnegative, and it is PSD and TN. We consider
functions defined on \([0,\infty)\); no continuity at zero is assumed.

\begin{theorem}\label{thm:zero}
Let \(f:[0,\infty)\to\mathbb R\).
\begin{enumerate}[label=\textup{(\roman*)}]
\item The matrix \(f[A]\) is PSD for every \(n\ge1\) and
\(A\in\mathcal C_n\) if and only if there is a finite positive Borel measure
\(\mu\) on \([0,\infty)\), possibly zero, such that
\[
f(t)=\int_{[0,\infty)}t^\alpha\,d\mu(\alpha),\qquad t>0,
\]
the integral is finite at every positive \(t\), and
\begin{equation}\label{eq:zero}
0\le f(0)\le\mu(\{0\}).
\end{equation}
\item The matrix \(f[A]\) is TN for every \(n\ge1\) and
\(A\in\mathcal C_n\) if and only if \(f\) has one of the following forms:
\begin{enumerate}[label=\textup{(\alph*)}]
\item \(f\equiv0\);
\item \(f(0)=0\) and \(f(t)=ct^\alpha\) for \(t>0\), where
\(c>0\) and \(\alpha>0\);
\item \(f(t)=c\) for every \(t\ge0\), where \(c>0\);
\item \(f(0)=0\) and \(f(t)=c\) for \(t>0\), where \(c>0\).
\end{enumerate}
\item The implication \(A\succeq B\Longrightarrow f[A]\succeq f[B]\)
holds for all \(n\ge1\) and \(A,B\in\mathcal C_n\) if and only if there is
a finite positive Borel measure \(\mu\) on \([0,\infty)\), possibly zero,
such that
\begin{equation}\label{eq:zero-order}
f(t)=f(0)+\int_{[0,\infty)}
\frac{t^{\alpha+1}}{\alpha+1}\,d\mu(\alpha),\qquad t>0,
\end{equation}
with the integral finite for every positive \(t\).
The number \(f(0)\) may be any real number.
\end{enumerate}
\end{theorem}

\begin{proof}
Suppose first that \(f\) preserves PSD on all the classes \(\mathcal C_n\).
Theorem~\ref{thm:PSD-only-preservers}, applied to the positive-entry inputs,
gives the representation in \textup{(i)}. Write \(d=f(0)\). The zero matrix
belongs to \(\mathcal C_n\), so \(d\ge0\). For \(t>0\),
\[
\begin{pmatrix}t&\varepsilon\\\varepsilon&\varepsilon\end{pmatrix}
\longrightarrow\diag(t,0),\qquad\varepsilon\downarrow0,
\]
and the approximants are positive definite Green matrices for
\(0<\varepsilon<t\). Thus \(\diag(t,0)\in\mathcal C_2\), and PSD
preservation gives
\[
\begin{pmatrix}f(t)&d\\d&d\end{pmatrix}\succeq0.
\]
Its determinant shows \(d\le f(t)\) when \(d>0\), and the inequality also
holds when \(d=0\). Dominated convergence in the power representation,
using \(t^\alpha\le1\) for \(0<t\le1\), gives
\(\lim_{t\downarrow0}f(t)=\mu(\{0\})\). Hence \eqref{eq:zero} follows.

For sufficiency, let \(A=(a_{ij})\in\mathcal C_n\). For every
\(\alpha>0\), continuity of the power map and
Proposition~\ref{prop:green-Hadamard-power} show that
\(A^{\circ\alpha}\) is PSD and TN, where \(0^\alpha=0\).
Define \(S(A)=(s_{ij})\) by \(s_{ij}=1\) if \(a_{ij}>0\) and
\(s_{ij}=0\) if \(a_{ij}=0\). Since
\[
S(A)=\lim_{\alpha\downarrow0}A^{\circ\alpha},
\]
this matrix is also PSD and TN. The entrywise identity
\begin{equation}\label{eq:boundary}
f[A]=d\mathbf1\mathbf1^T+
\bigl(\mu(\{0\})-d\bigr)S(A)
+\int_{(0,\infty)}A^{\circ\alpha}\,d\mu(\alpha)
\end{equation}
now proves PSD preservation, since all three terms are PSD.

Suppose next that \(f\) preserves TN on every \(\mathcal C_n\).
The output is symmetric and TN, hence PSD. Part \textup{(i)} therefore
applies, while Theorem~\ref{thm:TN-output-preservers} gives either zero or
\(f(t)=ct^\alpha\) on \((0,\infty)\), with \(c>0\) and \(\alpha\ge0\).
If the restriction is zero, \eqref{eq:zero} forces \(d=0\). If
\(\alpha>0\), its representing measure is \(c\delta_\alpha\), so again
\(d=0\). If \(\alpha=0\), part \textup{(i)} gives \(0\le d\le c\).
The matrices \((\e^{-T|i-j|})_{i,j=1}^3\) are strict Green and converge to
\(I_3\) as \(T\to\infty\). Thus \(I_3\in\mathcal C_3\), and the minor of
\(f[I_3]\) in rows \(1,2\) and columns \(2,3\) gives
\[
d(d-c)\ge0.
\]
Consequently \(d=0\) or \(d=c\), which proves necessity in
\textup{(ii)}. Conversely, the four possibilities give, respectively,
\(0\), \(cA^{\circ\alpha}\), \(c\mathbf1\mathbf1^T\), and \(cS(A)\).
Each is TN by the preceding argument.

For necessity in \textup{(iii)}, the restriction to positive arguments has
the representation in Theorem~\ref{thm:Loewner-green-fibers}. It has a finite
right limit \(a\) at zero, and
\[
f(t)=a+\int_{[0,\infty)}
\frac{t^{\alpha+1}}{\alpha+1}\,d\mu(\alpha).
\]
The matrices \(2tI_2\) and \(t\mathbf1\mathbf1^T\) belong to
\(\mathcal C_2\) and satisfy
\(2tI_2\succeq t\mathbf1\mathbf1^T\). Their output difference has diagonal
entries \(f(2t)-f(t)\) and off-diagonal entries \(f(0)-f(t)\).
Letting \(t\downarrow0\), positive semidefiniteness forces \(f(0)=a\).
This proves \eqref{eq:zero-order}.

Conversely, \eqref{eq:zero-order} makes \(f\) continuous on
\([0,\infty)\) and gives Loewner monotonicity on positive-parameter Green
matrices by Theorem~\ref{thm:Loewner-green-fibers}. Let
\(A\succeq B\) belong to \(\mathcal C_n\). If \(A_{ii}=0\), then
\(B_{ii}=0\), and both PSD matrices have zero \(i\)-th row and column.
The same row and column of \(f[A]-f[B]\) are zero. We may delete these
indices; principal submatrices remain in the corresponding closure class.
If no indices remain, then \(A=B=0\) and the conclusion is immediate.
We can therefore assume that every diagonal entry of \(A\) is positive.
For \(\varepsilon>0\), set
\[
A_\varepsilon=A\circ E_\varepsilon,\qquad
B_\varepsilon=B\circ E_\varepsilon,
\qquad E_\varepsilon=(\e^{-\varepsilon|i-j|})_{i,j=1}^n.
\]
The regularization in the proof of
Theorem~\ref{thm:Loewner-green-fibers}, applied to approximating Green
matrices, shows that \(A_\varepsilon,B_\varepsilon\in\mathcal C_n\).
It also shows that strict Green matrices are dense in \(\mathcal C_n\).
The Schur product theorem gives
\[
A_\varepsilon\succeq B_\varepsilon,
\qquad
A_\varepsilon\succeq
\lambda_{\min}(E_\varepsilon)
\diag(A_{11},\ldots,A_{nn})\succ0.
\]
For \(\eta>0\), it follows that
\[
(1+\eta)A_\varepsilon-B_\varepsilon\succeq\eta A_\varepsilon\succ0.
\]
Positive scaling preserves \(\mathcal C_n\), so both matrices in this
comparison can be approximated by strict Green matrices. The positive
definite difference ensures that sufficiently close approximants are still
comparable. Applying Theorem~\ref{thm:Loewner-green-fibers} to them and
passing to the limit gives
\[
f[(1+\eta)A_\varepsilon]\succeq f[B_\varepsilon].
\]
First let \(\eta\downarrow0\), then \(\varepsilon\downarrow0\).
Continuity of \(f\), including at zero, proves the required implication.
\end{proof}

The three classifications also hold if the inputs are restricted to PSD
matrices \(G(p,q)\) with finite nonnegative parameters \(p_i,q_i\).
To see the inclusion in \(\mathcal C_n\), delete every zero diagonal entry
and its zero row and column. On the remaining indices, the parameters are
positive and their ratios are nondecreasing. Extend these ratios to a
positive nondecreasing sequence on all indices, keep the original positive
\(q_i\)'s on the remaining indices, and set \(q_i=\varepsilon\) at the
deleted indices. Setting \(p_i=r_iq_i\) gives positive-parameter PSD Green
matrices converging to the original matrix. The zero matrix is treated by
\(\varepsilon\mathbf1\mathbf1^T\).

The necessity tests for \textup{(i)} already have finite nonnegative
parameters. For \textup{(ii)}, replace \(I_3\) by
\(\diag(0,1,0)=G((0,1,0),(1,1,0))\); the same cross minor is
\(f(0)(f(0)-c)\). For \textup{(iii)}, use
\[
A=t\begin{pmatrix}1&1\\1&2\end{pmatrix},
\qquad B=t\begin{pmatrix}0&0\\0&1\end{pmatrix},
\qquad A-B=t\mathbf1\mathbf1^T\succeq0.
\]
Both matrices have finite nonnegative Green parameters. As
\(t\downarrow0\), the output difference tends to
\(\left(\begin{smallmatrix}a-f(0)&a-f(0)\\a-f(0)&0\end{smallmatrix}\right)\),
which is PSD only if \(f(0)=a\). Thus the same necessary conditions hold
on the smaller class. The two input classes need not coincide; for example,
\(I_2\in\mathcal C_2\) has no finite nonnegative Green parameterization.

\section{Concluding remarks}\label{sec:conclusion}

The PSD-output condition allows positive mixtures of all nonnegative real
powers. Loewner monotonicity is determined by the same condition on the
derivative, even when the two Green matrices have different parameter
sequences. For TN output, a cross minor of a strict matrix of order three
forces multiplicativity; the remaining principal-minor tests reduce the
function to a nonnegative real power. Thus a single power is required both
for TN output and for preservation of PSD Green structure, although these
are different requirements on each individual image.

An order-two test does not determine PSD preservation in all orders.
For example, \(f(t)=\max(1,t)\) preserves PSD on positive-entry matrices of
order two: if \(ac\ge b^2\), then
\[
\max(1,a)\max(1,c)\ge\max(1,ac)\ge\max(1,b)^2.
\]
However, the positive rank-one Green matrix \(xx^T\), with
\(x=(1/2,1,2)^T\), satisfies
\[
f[xx^T]=\begin{pmatrix}1&1&1\\1&1&2\\1&2&4\end{pmatrix},
\qquad\det f[xx^T]=-1.
\]
Theorem~\ref{thm:fixed} gives the fixed-order classification through
geometric-progression tests, and Corollary~\ref{cor:euler} gives a pointwise
differential criterion for smooth functions. The strict hierarchy in
Corollary~\ref{cor:hierarchy} explains why the all-order power-mixture
representation cannot be imposed in any one fixed order. Passing to the
closure of the input class has the explicit
boundary condition \eqref{eq:zero}: the value at zero may be smaller than
the right limit. For TN output, Theorem~\ref{thm:zero} permits this jump
only for a positive constant on \((0,\infty)\), whereas Loewner
monotonicity forces continuity at zero by \eqref{eq:zero-order}.

\end{document}